\documentclass[12pt,oneside,english]{amsart}
\usepackage{amssymb,amsmath,latexsym,amsthm}

\usepackage[a4paper,top=3cm,bottom=3cm,left=3cm,right=3cm,marginparwidth=1.75cm]{geometry}
\usepackage{graphicx}

\usepackage{xcolor}

\newtheorem{theorem}{Theorem}[section]
\newtheorem{corollary}[theorem]{Corollary}
\newtheorem{lemma}[theorem]{Lemma}
\newtheorem{prop}[theorem]{Proposition}

\theoremstyle{definition}
\newtheorem{definition}[theorem]{Definition}
\newtheorem{remark}[theorem]{Remark}
\newtheorem{example}[theorem]{Example}

 \newtheorem*{thmA}{Theorem A} 
 \newtheorem*{thmB}{Theorem B} 
 \newtheorem*{thmC}{Theorem C} 
\newtheorem*{thmD}{Theorem D}

\newcommand\blfootnote[1]{%
  \begingroup
  \renewcommand\thefootnote{}\footnote{#1}%
  \addtocounter{footnote}{-1}%
  \endgroup
}
\newcommand{\PSH}{{\rm PSH}}

\newcommand{\capa}{{\rm Cap}}

\numberwithin{equation}{section}

 \usepackage{hyperref}
\hypersetup{  
    pdftitle={Cegrell classes},    
    pdfauthor={Mohammed Salouf},     
    colorlinks=true,       
   linkcolor=black,          
    citecolor=black,        
    filecolor=black,      
    urlcolor=black}

\subjclass[2010]{32W20, 32U05, 32Q15, 35A23}

\keywords{Monge-Amp\`ere type equations, Cegrell's classes}

\begin{document}

 \title[Degenerate complex Monge-Amp\`ere type equations]{Degenerate complex Monge-Amp\`ere  equations with non-K\"ahler forms in  bounded domains }
\author{Mohammed Salouf}

\address{ 
Department of Mathematics, 
Faculty of Sciences El Jadida, 
Chouaib Doukkali University.
24000 El Jadida. Morocco.} 
\email{salouf.m@ucd.ac.ma}

\date{\today}
   
\maketitle

\begin{abstract}
In this paper, we study weak  solutions to  complex Monge-Amp\`ere equations of the form  $(\omega + dd^c \varphi)^n= F(\varphi,.)d\mu$ on a bounded strictly pseudoconvex domain in $\mathbb{C}^n$, where $\omega$ is a smooth $(1,1)$-form, $0\leq F$ is a continuous non-decreasing function, and $\mu$ is a positive non-pluripolar measure. Our results extend previous works of Ko{\l}odziej and Nguyen \cite{KN15,KN23a,KN23b} who study bounded solutions, as well as Cegrell \cite{Ceg98,Ceg04,Ceg08}, Czy\.z \cite{Cz09}, Benelkourchi \cite{Ben09,Ben15} and others who treat the case when $\omega=0$ and/or $F=1$.    
\end{abstract}

\blfootnote{The author is supported by the CNRST within the framework of the Excellence Research Grants Program under grant number 18 UCD2022.}
\section{Introduction}

Let $\Omega$ be a bounded strictly pseudoconvex domain in $ \mathbb{C}^n$. By definition, there exists a smooth strongly plurisubharmonic function $\rho$ defined in a neighborhood of $\bar{\Omega}$ such that $d\rho \neq 0$ on $\partial \Omega$ and $\Omega = \{\rho<0\}$.

The Monge-Amp\`ere operator is defined on smooth functions $u$ by the formula 
$$ (dd^c u)^n := 4^n n! \det \left( \frac{\partial^2 u}{\partial z_j \partial \bar{z}_k} \right) dV_n. $$
When $u$ is plurisubharmonic, i.e. $dd^c u$ is a positive $(1,1)$-form, the above defines a smooth volume form.  In the foundational work \cite{BT76}, Bedford and Taylor succeeded in defining $(dd^c u)^n$ as a positive Radon measure for all locally bounded plurisubharmonic functions $u$. As is well-known extending this operator to unbounded psh functions is a delicate task.  Cegrell introduced in \cite{Ceg98,Ceg04,Ceg08} several classes of (unbounded) plurisubharmonic functions $u$, for which the Monge-Amp\`ere operator $(dd^c .)^n$ is well defined and enjoys natural convergence properties. In this series of papers, he gave detailed study of the solutions to the Dirichlet problem.  In  \cite{Ben09,Ben15} Benelkourchi studied weighted energy classes in the spirit of \cite{GZ07} and provided a complete description of solutions to the Dirichlet problem for weights that are convex or homogeneous.  Our first main result provides the same description for concave weights $\chi$ that have polynomial-like behavior, i.e. satisfying 
$$ -t \chi'(t) \leq -M\chi(t) \; \; \forall t \in \mathbb{R}^-, $$
with a uniform positive constant $M$. We let $\mathcal{W}^+_M$ denote the set of such weights.  

\newpage

\begin{thmA}[Theorem \ref{thm 3.5}]
{\it    Let $\mu$ be a positive Radon measure, and let $\chi \in \mathcal{W}^+_M$. The following conditions are equivalent:
\begin{itemize}
    \item[(1)]  there exists a unique function $ \phi \in \mathcal{E}_\chi(\Omega)$ such that $ \mu = (dd^c
\phi)^n$;
\item[(2)]$\chi(\mathcal{E}_\chi(\Omega)) \subset L^1(d\mu)$;
\item[(3)]  there exists a constant $C > 0$ such that
$$ \int_\Omega - \chi \circ \psi d\mu \leq C, $$ for all  $\psi \in \mathcal{E}_0(\Omega)$, 
$  
E_\chi(\psi) \leq 1; $
\item[(4)] there exists a positive constant $A$ such that
$$ \int_\Omega -\chi \circ \psi d\mu \leq A \max (1, E_\chi(\psi)), \; \; \forall \psi \in \mathcal{E}_0(\Omega).
$$
\end{itemize}
}
\end{thmA}
The equivalence (1) $\Leftrightarrow$ (2) in the last theorem was conjectured in \cite{GZ07} in the case of compact K\"ahler manifolds. It has been recently solved positively in \cite{TV21} and \cite{DDL23}.  Our proof uses ideas from \cite{DDL23} where plurisubharmonic envelopes play a crucial role.  

Next, we turn our attention to solutions of complex Monge-Ampère type equations: giving  a positive Radon measure $\mu$ vanishing on pluripolar sets and a bounded measurable function  $F: \mathbb{R} \times \Omega \rightarrow [0,+\infty[$ which is continuous and non-decreasing  in the first variable,
 we are interested in 
 the study of the following equation 
\begin{equation}\label{eq: CMATE}
    (\omega+dd^c \varphi)^n = F(\varphi,.)d\mu.
\end{equation}
Here $\omega$ is a smooth $(1,1)$-form defined on a neighborhood of $\bar{\Omega}$. We also stress that we do not assume $\omega$ is closed. The Dirichlet problem for the Monge-Ampère operator $(dd^c .)^n$ corresponds to the case when $F=1$ and $\omega=0$.

Bounded solutions to \eqref{eq: CMATE} have been studied by S. Ko{\l}odziej and N.C. Nguyen in \cite{KN15,KN23a,KN23b}. To study unbounded solutions, we first introduce natural generalizations of Cegrell's classes.
Let $\phi \in \mathcal{E}(\Omega) \cap \mathcal{C}^0(\bar{\Omega})$ be a maximal function.
For $ \mathcal{K}(\Omega) \in \{ \mathcal{E}(\Omega)$, $\mathcal{F}(\Omega)$, $\mathcal{E}_p(\Omega)$, $\mathcal{E}_\chi(\Omega)$,
$\mathcal{N}(\Omega)$\},
 we define the class $\mathcal{K}(\Omega,\omega,\phi)$ by:
$$ u \in  \mathcal{K}(\Omega,\omega, \phi) \Leftrightarrow u \in \PSH(\Omega,\omega) \; \text{and} \;  u + \rho \in \mathcal{K}(\Omega,\phi). $$
Here $ \rho $ is a plurisubharmonic  function of class $\mathcal{C}^2$ in a neighborhood of $ \bar{\Omega}$ such that $\rho = 0$ on $\partial \Omega$ and  $\omega \leq  dd^c \rho$. The set $\mathcal{K}(\Omega,\omega)$ corresponds to $\mathcal{K}(\Omega,\omega,\phi)$ with $\phi = 0$. 
When $\omega = dd^c \rho$, the sets $ \mathcal{F}(\Omega,\omega)$ and $\mathcal{E}_\chi(\Omega,\omega)$ coincide with their counterparts  given in \cite{CKZ11}.

The following result asserts that the Monge-Ampère  operator $(\omega + dd^c .)^n$ is well defined on the classes $\mathcal{K}(\Omega,\omega,\phi)$. 
 \begin{thmB}[Theorem \ref{thm 4.10}]
{\it  
 The Monge-Ampère measure  $(\omega + dd^c u)^n$ is well  defined
 as a Radon measure  in $\Omega$
 for all $u \in \mathcal{E}(\Omega,\omega)$. Furthermore, if $(u_j)_j$ is a decreasing sequence in $\mathcal{E}(\Omega,\omega)$ that converges to $u \in \mathcal{E}(\Omega,\omega)$, then the sequence $((\omega + dd^c u_j)^n)_j$ converges weakly to $(\omega +dd^c u)^n$.}
 \end{thmB}

The direct adaptation of Cegrell's proof breaks down in our context because the integration by parts is missing. Indeed, since the reference form is not closed,  applying Stokes theorem produces several torsion terms which are difficult to control even for bounded potentials. Moreover, Example \ref{ex 5.1} below shows that the operator $(\omega + dd^c .)^n$ may fail to verify the basic properties of the operator $(dd^c .)^n$.   

 To prove Theorem B, we decompose $\omega$ into a finite sum of $(1,1)$-forms of the form $f \alpha$, where $f$ is a smooth function and $\alpha$ is a  positive closed $(1,1)$-form. Then $(\omega + dd^c .)^n$  can be written as the sum of terms involving mixed Monge-Ampère operators.

 After defining the operator $(\omega + dd^c .)^n$, a natural question to ask is whether there are solutions to  \eqref{eq: CMATE} in the classes $\mathcal{K}(\Omega,\omega,\phi)$.
 This equation appeared for the first time in the problem of constructing K\"ahler-Einstein metrics in compact K\"ahler manifolds. When $\omega=0$, Bedford and Taylor obtained bounded solutions to this equation in strictly pseudoconvex domains of $\mathbb{C}^n$  \cite{BT79}. This result has been generalized in many context 
\cite{Kol95,CK06,ACCH08,Cz09,HH11,Ben13}.

Recently, Ko{\l}odziej and Nguyen extended this problem to the operator $(\omega + dd^c .)^n$. They proved the existence of bounded solutions   to \eqref{eq:  CMATE} when the measure $\mu$ is dominated by the Monge-Ampère measure of some function $v \in \mathcal{E}_0(\Omega)$ \cite[Theorem 3.1]{KN23b}.
Here we extend this result to more singular measures $\mu$ for which  we seek for unbounded weak solutions.

\begin{thmC}[Theorem \ref{thm 5.8}]
 { \it  Assume  $\mu\leq (dd^c u)^n$ is dominated by  the Monge-Ampère measure of some function $u \in \mathcal{K}(\Omega)$.
      Then  there is a uniquely determined $\varphi \in \mathcal{K}(\Omega,\omega,\phi)$  solving \eqref{eq: CMATE}.    }
\end{thmC}
In relation to this theorem, several results have been obtained in the case $\omega=0$. 
Ko\l odziej has shown that the equation $\mu = (dd^c u)^n$ has a unique bounded solution $u$ if $\mu$ is dominated by the Monge-Ampère measure of a bounded psh function \cite{Kol95}. Based on this result, Cegrell  characterized the range of the operator $(dd^c .)^n$ on $\mathcal{K}(\Omega)$ \cite{Ceg98,Ceg04,Ceg08}.  Later on, Ahag, Cegrell, Czy{\.z} and Hiep established a subsolution theorem in $\mathcal{K}(\Omega)$ \cite{ACCH08}. They proved that the equation $\mu = (dd^c u)^n$ has a solution in $\mathcal{E}(\Omega)$ if $\mu$ is dominated by the Monge-Ampère measure of a function belonging to $\mathcal{E}(\Omega)$. In \cite{Cz09},
Czy{\.z} proved that  the equation $(dd^c u)^n = F(.,u) d\mu$ has a unique solution $u \in \mathcal{N}^a(\Omega,\phi)$ if $\mu$ is the Monge-Ampère measure of some function in $\mathcal{N}^a(\Omega)$.    When $\omega>0$, bounded solution to \eqref{eq: CMATE} have been obtained by Ko\l odziej and Nguyen \cite{KN15,KN23a,KN23b}.

The uniqueness part in the previous theorem  is a consequence of  the following comparison principle that generalizes \cite[Theorem 5.15]{Ceg04}, \cite[Theorem 4.4]{Ceg08},  \cite[Corollary 3.4]{KN15} and \cite[Proposition 2.2]{KN23b}.
\begin{thmD}[Corollary \ref{cor 5.7}]
{\it Let $\mu \leq \nu$ be positive Radon measures vanishing on pluripolar  sets. 
Assume $u \in \mathcal{N}(\Omega,\omega,\phi)$ and  $v \in \mathcal{E}(\Omega,\omega)$ are such that  $v \leq \phi$ in $\partial \Omega$,
$$ (\omega + dd^c u)^n = F(u,.) d\mu \; \text{and} \;  (\omega + dd^c v)^n = F(v,.) d\nu. $$
 Then $u \geq v$. } 
\end{thmD}
Let us explain briefly the idea of the proof of the comparison principle which is inspired by \cite{LN22}. 
We consider the following plurisubharmonic envelope 
$$ P(u-v) := \sup \{  \varphi \in \PSH^-(\Omega): \varphi \leq u-v \},$$
assuming, without loss of generality, that $u\leq v$. 
Using that the Monge-Ampère measure of the envelope is concentrated on the contact set $\{P(u-v) = u-v\}$, we show that $(dd^c P(u-v))^n=0$, hence $u=v$. 

The paper is organized as follows. 
In Section \ref{sec 2}, we recall some properties of  the Cegrell classes  that we shall use in the sequel. After, we move on to the study of the plurisubharmonic envelopes.  Section \ref{sec 3}  is devoted to the study of the Dirichlet problem in the weighted energy class $\mathcal{E}_\chi(\Omega)$ for $\chi \in \mathcal{W}^+_M$. In Section \ref{sec 4}, we prove that the operator $(\omega + dd^c .)^n$ is well defined on the large set $\mathcal{E}(\Omega,\omega)$ and that it is continuous along decreasing sequences.
The proof of Theorem C and Theorem D will be the subject of Section \ref{sec 5}.

Throughout this paper,  $\Omega$ is a  bounded strictly pseudoconvex domain of $\mathbb{C}^n$ and $n \geq 1$. 

\section*{Acknowledgment}
I would like to express my gratitude to my supervisors Omar Alehyane  and Chinh H. Lu for introducing the subject, for their generosity in providing knowledge and expertise and for all that time spent. The author is  grateful to the referee for valuable suggestions that improve the presentation of the paper.
\section{Cegrell's classes}\label{sec 2}
In this section, we first present a brief introduction to the Cegrell classes, and then we 
study  plurisubharmonic envelopes.
\subsection{The classes $\mathcal{E}_0$, $\mathcal{E}$, $\mathcal{E}_p$, $\mathcal{F}$ and $\mathcal{N}$}
In \cite{Ceg98,Ceg04}, Cegrell introduced the following classes which carry his name:
$$ \mathcal{E}_0(\Omega) = \{ u \in \PSH(\Omega)\cap L^{\infty}(\Omega) : u = 0 \; \text{on} \; \partial{\Omega} \; \text{and} \; \int_\Omega (dd^c u)^n < +\infty    \},
$$
\begin{flalign*}
 \mathcal{E}(\Omega) &= \{ u \in \PSH^-(\Omega) : \forall z \in \Omega, \exists V \in \mathcal{V}(z), \exists (u_j)_j \subset \mathcal{E}_0(\Omega),\\
 &\; u_j \searrow u \; \text{on}\; V \;  \text{and} \; \sup_j  \int_\Omega (dd^c u_j)^n < +\infty    \}, 
 \end{flalign*}
$$ \mathcal{F}(\Omega) = \{ u \in \PSH^-(\Omega) : \exists u_j \in \mathcal{E}_0(\Omega), \; u_j \searrow u \; \text{and} \; \sup_j \int_\Omega (dd^c u_j)^n < +\infty    \}, $$
and for $p>0$,
$$ \mathcal{E}_p(\Omega) = \{ u \in \PSH(\Omega) : \exists (u_j)_j \subset \mathcal{E}_0(\Omega), \; u_j \searrow u \; \; \text{and} \; \sup_j \int_\Omega   |u_j|^p (dd^c u_j)^n <+\infty  \}. $$
We have the inclusions 
$$ \mathcal{E}_0(\Omega) \subset \mathcal{E}_p(\Omega) \cap  \mathcal{F}(\Omega) \subset \mathcal{E}_p(\Omega) \cup  \mathcal{F}(\Omega) \subset \mathcal{E}(\Omega). $$
If $u \in \mathcal{E}(\Omega)$, then $(dd^cu)^n$ defines a positive Radon measure by \cite[Theorem 4.2]{Ceg04}. 
The set $\mathcal{E}(\Omega)$ is the largest set for which the Monge-Ampère operator $(dd^c.)^n$ is well defined and continuous along decreasing sequences \cite[Theorem 4.5]{Ceg04}. 

We recall the class $\mathcal{N}(\Omega)$ defined in \cite{Ceg08}. Let $u \in \mathcal{E}(\Omega)$, and let $(\Omega_j)$  be a fundamental sequence of strictly pseudoconvex subdomains
of $\Omega$. Define
$$ u_j = \sup \{ \varphi \in \PSH^-(\Omega) : \varphi \leq u \; \text{on} \; \mathcal{C}\Omega_j \}. $$
Note that we have $u \leq u_j \leq u_{j+1}$ for every $j$. The class $\mathcal{N}(\Omega)$ is the set of $u\in \mathcal{E}(\Omega)$ such that $\tilde{u} := (\lim u_j)^* =0$. We have 
$$ \mathcal{F}(\Omega) \subset \mathcal{N}(\Omega) \; \text{and} \; \mathcal{E}_p(\Omega) \subset \mathcal{N}(\Omega), \; \forall p. $$
 We denote by   $\mathcal{E}^a(\Omega)$ the set of $u \in \mathcal{E}(\Omega)$ such that $(dd^c .)^n$ vanishes on pluripolar  sets.
The following  result is known as the comparison principle.
\begin{theorem}[Theorem 3.12 and Corollary 3.13 in \cite{Ceg08}]
Let $u \in \mathcal{N}^a(\Omega)$ and let $v \in \mathcal{E}(\Omega)$. We have 
$$ \int_{\{u<v\}} (dd^c v)^n \leq \int_{\{u<v\}} (dd^c u)^n. $$
In particular, if $(dd^c u)^n \leq (dd^c v)^n$ then $u\geq v$.
\end{theorem}
The following theorem gives an idea about the range of the Monge-Ampère operator on $\mathcal{N}^a(\Omega)$. 
\begin{theorem}[Proposition 5.2 in \cite{Ceg08}]
Let $\mu$ be a positive Radon measure vanishing on pluripolar sets. Suppose there is $\psi \in \mathcal{E}(\Omega)$ with $\psi \neq 0$ and $\int_\Omega \psi d\mu>-\infty$.  Then there is a uniquely determined $u\in \mathcal{N}^a(\Omega)$ such that 
$$ (dd^c u)^n = \mu. $$
\end{theorem}
Note that the converse of this theorem is not true as Cegrell showed in \cite[Example 5.3]{Ceg08}. 
\subsection{Plurisubharmonic envelopes} 
This subsection is devoted to the study of the following plurisubharmonic envelope: For a  measurable function $f$, the envelope $P(f)$ is defined by 
$$ P(f) = \left( \sup \{ \varphi \in \PSH(\Omega) : \varphi \leq f \}\right)^*. $$
The study of this object  has attracted the interest of several authors over the last decade (see \cite{BT82,DDNL18,GLZ19,Do20}, and the references therein for more information). 
We shall use this  envelope to prove a general comparison principle (see the proof of Theorem \ref{thm 5.6}), and to prove Theorem A in the introduction. 

 First, we should prove the following proposition. 

\begin{prop}\label{prop 2.3}
    If $f$ is a measurable  function, then 
    $$ P(f) = \sup  \{ \varphi \in \PSH(\Omega) : \varphi \leq f \; \text{quasi-everywhere} \; \}, $$
    where the term quasi-everywhere means outside a pluripolar set. In particular if $(f_j)$ is a decreasing sequence of measurable  functions converging to $f$, then $P(f_j)$ decreases to $P(f)$.
\end{prop}
\begin{proof}
    Let us denote by $h$ the function on the right hand side. We have to prove that $P(f) = h$. 
    It follows from \cite[Proposition 5.1]{BT82} that $P(f) \leq f$ quasi-everywhere and hence $P(f) \leq h$. In the other hand, by Choquet's lemma, there is $\varphi_j \in \PSH^-(\Omega)$ such that $\varphi_j \leq f$ quasi-everywhere  and $h^* = (\sup \varphi_j)^*$. Since countable union of pluripolar sets is pluripolar, it follows that 
 $h^* \leq f$ quasi-everywhere and hence  $h = h^* \in \PSH(\Omega)$. Since $h \leq f$ quasi-everywhere, there is $\phi \in \PSH^-(\Omega)$ such that $h+ \varepsilon \phi \leq f$ everywhere for all $\varepsilon>0$. It follows that $h+\varepsilon \phi \leq P(f)$. Letting $\varepsilon \rightarrow 0$, we get $h \leq P(f)$ quasi-everywhere and hence everywhere because these are psh functions. 
 We conclude that $P(f) = h$. 

 Let $(f_j)$ be a decreasing sequence of  measurable
 functions converging to $f$. Obviously, the sequence $(P(f_j))_j$ is decreasing and $P(f) \leq P(f_j)$ for all $j$. In the other hand, we have $P(f_j) \leq f_j$ quasi-everywhere for every $j$. Therefore $\lim P(f_j) \leq f$ quasi-everywhere and hence $\lim P(f_j)=P(f)$.
\end{proof}
The following theorem is a generalization of \cite[Corollary 9.2]{BT82} to quasi-continuous functions $f$.     
\begin{theorem}\label{thm 2.4}
     Assume $f$ is  quasi-continuous, $f\leq 0$, and 
      there is $\psi\in \mathcal{E}^a(\Omega)$ such that $\psi\leq f$. Then  $P(f)\in \mathcal{E}^a(\Omega)$ and $(dd^c P(f))^n$ is concentrated on $\{P(f)=f\}$. 
\end{theorem}
Before proving the theorem, we should recall the definition of the Monge-Ampère capacity defined in \cite{BT82}: Giving a Borel subset $E \subset \Omega$, we set
$$ \capa(E) := \sup \left\{ \int_E (dd^c u)^n  : \; u \in \PSH(\Omega), \; -1 \leq u \leq 0 \right\}.  $$
A function $f$ is called quasi-continuous if for every $\varepsilon>0$, there is a Borel set $E \subset \Omega$ such that $\capa(E) \leq \varepsilon$ and the restriction of $f$ on $\Omega \setminus E$ is continuous. 

We now proceed to the proof of Theorem \ref{thm 2.4}.
\begin{proof}
Assume first that $f$ is bounded from below.  
For each $j\geq 1$, there is an open set $U_j\subset \Omega$ such that $\capa(U_j)\leq 2^{-j-1}$ and the restriction of $f$ on $\Omega\setminus U_j$ is continuous. By taking $\cup_{k\geq j} U_k$ we can assume that the sequence $U_j$ is decreasing. By the Tietze extension theorem, there is a function $f_j$ continuous on $\Omega$ such that $f_j=f$ on $D_j:=\Omega\setminus U_j$. We can assume that there is a constant $C_0$ such that $-C_0\leq f_j\leq 0$, for all $j$. For each $j$ we define 
\[
g_j:= \sup_{k\geq j} f_k.
\]
We observe that $g_j$ is lower-semicontinuous in $\Omega$, $g_j=f$ on $D_j$, and $g_j\searrow g$ in $\Omega$. Since the sequence $(D_j)$ is increasing, it follows that $g_j=f$ on $D_k$ for all $k\leq j$. Thus letting $j\to+\infty$ gives $g=f$ on $D_k$ for all $k$. We then infer $g=f$ quasi-everywhere in $\Omega$, hence $P(f)=P(g)$ by Proposition \ref{prop 2.3}.  Since $g_j$ is lower-semicontinuous in $\Omega$, by the balayage method, \cite[Corollary 9.2]{BT82}, we have
\[
\int_{\Omega} (g_j - P(g_j))(dd^c P(g_j))^n=0. 
\]
From this we get 
\begin{flalign*}
\int_{\Omega} |f-P(g_j)| &(dd^c P(g_j))^n \\ 
&= \int_{D_j} |f-P(g_j)|(dd^c P(g_j))^n
+ \int_{U_j} |f-P(g_j)|(dd^c P(g_j))^n\\
& = \int_{D_j} (g_j-P(g_j))(dd^c P(g_j))^n+ \int_{U_j} |f-P(g_j)|(dd^c P(g_j))^n\\
&\leq  2C_0\int_{U_j}  (dd^c P(g_j))^n\\
&\leq 2(C_0)^{n+1} \int_{U_j} (dd^c P(g_j)/C_0)^n\\
&\leq 2(C_0)^{n+1} \capa(U_j)\leq (C_0)^{n+1}2^{-j}. 
\end{flalign*}
The functions $|f-P(g_j)|$ are uniformly bounded and converge in capacity to the quasi-continuous function $f-P(f)$ because
\[ 
\left||f - P(g_j)|- f + P(f)\right| \leq |P(g_j) - P(f)|, 
\]
and the sequence $(P(g_j))_j$ decreases to $P(f)$.
It thus follows from \cite[Theorem 4.26]{GZ17} that $|f-P(g_j)|(dd^c P(g_j))^n$ weakly converges to $(f-P(f))(dd^c P(f))^n$. Hence 
\[
0 = \liminf_{j\to +\infty} \int_{\Omega} |f-P(g_j)|(dd^c P(g_j))^n \geq \int_{\Omega} (f-P(f))(dd^c P(f)^n\geq 0. 
\]
From this we infer that $(dd^c P(f))^n$ is concentrated on the contact set $\{P(f)=f\}$. 

To prove the general case we approximate $f$ by $f_j:=\max(f,-j)$. Then 
\[
\int_{\{P(f_j)<f_j\}} (dd^c P(f_j))^n=0. 
\]
Fixing $C>0$,  we have  by \cite[Theorem 4.1]{KH09}
\[
\int_{\{P(f_j)<f_j\}\cap \{P(f) >-C\}} (dd^c \max(P(f_j),-C))^n=0. 
\]
 Fixing an integer $k \in \mathbb{N}$, we have 
 \[
\int_{\{P(f_k)<f\}\cap \{P(f)>-C\}} (dd^c \max(P(f_j),-C))^n=0 \; \forall j \geq k, 
\]
because in this case $\{P(f_k)<f\} \subset \{P(f_j)<f_j\}$. Set 
$$ h_k = \left(\max(f,P(f_k)) - P(f_k)\right) \times \left( \max(P(f),-C) + C \right). $$
The function $h_k$ is positive, bounded, quasi-continuous, and satisfies 
 \[
\int_\Omega h_k (dd^c \max(P(f_j),-C))^n=0 \; \; \forall j \geq k.
\]
 Letting $j\to +\infty$ we obtain again by \cite[Theorem 4.26]{GZ17} 
\[
\int_{\{P(f_k)<f\}\cap \{P(f)>-C\}} (dd^c \max(P(f),-C))^n=0. 
\]
Next, letting $k\to +\infty$ we arrive at 
\[
\int_{\{P(f)<f\}\cap \{P(f)>-C\}} (dd^c \max(P(f),-C))^n=0. 
\]
By \cite[Theorem 4.1]{KH09} we then have 
\[
\int_{\{P(f)<f\}\cap \{P(f)>-C\}} (dd^c P(f))^n=0. 
\]
We finally let $C\to +\infty$ to obtain the result since $(dd^c P(f))^n$ does not charge pluripolar sets. 
\end{proof}
\begin{remark}
    We use the hypothesis $\psi \leq f$, for certain $\psi \in \mathcal{E}^a(\Omega)$, to ensure that $P(f) \in \mathcal{E}^a(\Omega)$. If $P(f) \in \mathcal{E}(\Omega)$ charges the pluripolar set $\{P(f)=-\infty\}$, then the same proof shows that the measure $(dd^c P(f))^n$ vanishes on  $\{P(f)<f\} \cap \{P(f) > -\infty\}$. 
\end{remark}
\section{High energy classes}\label{sec 3}
In this section, we study the existence of solutions to the following Dirichlet problem
\begin{equation}\label{CMAE}
    (dd^c u)^n = \mu \quad u \in \mathcal{E}_\chi(\Omega),
\end{equation}
 where $\mu$ is a positive Radon measure vanishing on pluripolar  sets. The equation \eqref{CMAE} has been studied by Benelkourchi \cite{Ben09,Ben15} in the case of convex or homogeneous weights. 
We extend these results to a special type of concave functions $\chi$.

First, we recall the class $\mathcal{E}_\chi(\Omega)$ defined in \cite{GZ07}; we denote by $\mathcal{W}^-$ (resp $\mathcal{W}^+$) the set of convex (resp concave) increasing functions  $\chi : \mathbb{R}^- \rightarrow \mathbb{R}^-$ such that $\chi(-\infty)=-\infty$. The set $\mathcal{W}^+_M$ consists of functions $\chi \in \mathcal{W}^+$ with the property 
$$ |t \chi'(t)| \leq M |\chi(t)|, \; \; \forall t \in \mathbb{R}^-. $$

Let $\chi \in \mathcal{W} := \mathcal{W}^- \cup \mathcal{W}^+$. The set $\mathcal{E}_\chi(\Omega)$ is defined by 
$$ \mathcal{E}_\chi(\Omega) = \left\{ u \in \PSH(\Omega) : \exists (u_j)_j \subset \mathcal{E}_0(\Omega), \; u_j \searrow u \; \; \text{and} \; \sup_j \int_\Omega -\chi \circ u_j (dd^c u_j)^n <+\infty  \right\}. $$
For $u\in \mathcal{E}_\chi(\Omega)$, we use the notation 
$$ E_\chi(u) = \int_\Omega -\chi(u) (dd^c u)^n. $$

The following lemma will be very helpful in the sequel.
\begin{lemma}[Lemma 2.2 in \cite{TV21}] \label{lem 3.1}
If $\chi \in \mathcal{W}^+_M$, then
$$ -\chi(ct) \leq -c^M\chi(t), $$
for all $t\leq 0$ and all $c\geq 1$.
\end{lemma}
The last lemma allows us to prove the following proposition:
\begin{prop}
Fix $\chi \in \mathcal{W}^+_M$.  
\begin{itemize}
    \item[(i)] We have $\chi(t) <0$ for all $t<0$; in particular
    $$ \mathcal{E}_\chi(\Omega) = \{u \in \mathcal{N}^a(\Omega) : \chi(u) \in L^1((dd^cu)^n) \}. $$
\item[(ii)] The set  $\mathcal{E}_\chi(\Omega)$ is a convex cone.
\end{itemize}

\end{prop}
\begin{proof} $\phantom{}$
\begin{itemize}
    \item[(i)] By contradiction, assume that there is $t_0<0$ such that $\chi(t_0)=0$. If $t_0 \leq -1$ then $\chi(-1) = 0$. If $t_0 >-1$ then 
$$ -\chi(-1) = -\chi(t_0/|t_0|) \leq \frac{1}{|t_0|^M} \chi(t_0) = 0. $$
It follows that $\chi(-1) = 0$ in both cases. Let $t\in \mathbb{R}_-$. We have 
$$ -\chi(t) = -\chi((-t) \times (-1)) \leq - \max(|t|^M,|t|)\chi(-1) = 0. $$
We conclude that $\chi = 0$ and this is absurd because $\chi(-\infty) = -\infty$. 

The second affirmation follows from \cite[Corollary 3.3]{HH11}.
\item[(ii)]
The set $\mathcal{E}_\chi(\Omega)$ is convex by \cite[Proposition 4.3]{Ben09}. It suffices to prove that if $u\in \mathcal{E}_\chi(\Omega)$ then so is $2u$. This statement follows easily from the last lemma:
$$ E_\chi(2u) = 2^n \int_\Omega -\chi(2u) (dd^c u)^n \leq 2^{n+M} E_\chi(u). $$
\end{itemize}
\end{proof}
In \cite[Definition 4.1]{BGZ08}, the authors defined the following class for an increasing function $\chi$
$$ \tilde{\mathcal{E}}_\chi(\Omega) = \{ u \in \PSH^-(\Omega) : \int_0^{+\infty} t^n \chi'(-t) \capa_\Omega(u<-t) dt <+\infty \}. $$
They proved the inclusions 
$$ \tilde{\mathcal{E}}_\chi(\Omega) \subset \mathcal{E}_\chi(\Omega) \subset \tilde{\mathcal{E}}_{\tilde{\chi}}(\Omega),  $$
where $\tilde{\chi}(t) =\chi(t/2)$ \cite[Proposition 4.2]{BGZ08}. We have the following observation:
\begin{prop}\label{prop 3.3}
    If $\chi \in \mathcal{W}^+_M$, then $\tilde{\mathcal{E}}_\chi(\Omega) = \mathcal{E}_\chi(\Omega)$. 
\end{prop}
\begin{proof}
    We have 
    $$ \mathcal{E}_\chi(\Omega) \subset \tilde{\mathcal{E}}_{\tilde{\chi}}(\Omega) \subset  \mathcal{E}_{\tilde{\chi}}(\Omega). $$
    So it suffices to prove that $\mathcal{E}_{\tilde{\chi}}(\Omega) \subset \mathcal{E}_\chi(\Omega)$. Let $u \in \mathcal{E}_{\tilde{\chi}}(\Omega)$. We have
  $$ E_\chi(u) = \int_\Omega -\chi(u) (dd^cu)^n \leq 2^M \int_\Omega -\chi(u/2) (dd^c u)^n = 2^M E_{\tilde{\chi}}(u) <+\infty. $$
\end{proof}
\begin{theorem}\label{thm 3.4}
Fix $\chi \in \mathcal{W}^+_M$, and let $u, v \in \mathcal{E}_\chi(\Omega)$. We have 
$$ \int_\Omega - \chi \circ u (dd^c v)^n \leq \lambda^{-n} E_\chi(2\lambda v) + 2^M\lambda^{-n}E_\chi(u),$$ for all $\lambda>0$.
\end{theorem}
\begin{proof} The proof uses ideas from \cite[Theorem 5.1]{Ben09}.

Fix $\lambda>0$.
Using the fact that 
$$ (u < -t) \subset (u < \lambda v - t/2) \cup ( \lambda v < - t/2), $$
we get 
\begin{align*}
    \lambda^n \int_\Omega -\chi(u) (dd^c v)^n
    &=\int_\Omega -\chi(u)(dd^c \lambda v)^n \\ 
    &= \int_0^{+\infty}  (dd^c \lambda v)^n (\chi \circ u<-t) dt \\
    &= \int_0^{+\infty} \chi'(-t) (dd^c \lambda  v)^n(u<-t) dt\\
    &\leq \int_0^{+\infty} \chi'(-t) (dd^c \lambda v)^n(u<\lambda v-t/2) dt \\ &+ \int_0^{+\infty} \chi'(-t) (dd^c \lambda v)^n(\lambda v<-t/2) dt.
\end{align*}

In one hand
 $$\int_0^{+\infty} \chi'(-t) (dd^c\lambda v)^n(\lambda v<-t/2) dt \leq  \int_0^{+\infty} \chi'(-t) (dd^c 2\lambda  v)^n(2\lambda  v<-t)dt = E_\chi(2\lambda v). $$

 In the other hand, we have  by \cite[Corollary 3.13]{Ceg08}
 \begin{align*}
     \int_0^{+\infty} \chi'(-t) (dd^c \lambda  v)^n(u<\lambda  v-t/2) dt  &\leq \int_0^{+\infty} \chi'(-t) (dd^c u)^n(u<\lambda  v-t/2) dt \\
     &\leq \int_0^{+\infty} \chi'(-t) (dd^c u)^n(u<-t/2) dt \\
     &\leq \int_\Omega -\chi(2u) (dd^c u)^n \\
     &\leq 2^M  E_\chi(u).
 \end{align*}
\end{proof}

The following result 
corresponds to
\cite[Theorem 6]{Ben15} in the case of convex or homogeneous weight $\chi$ (see also \cite[Theorem A]{ACC12} for the case $\chi(t) = t$).
\begin{theorem}\label{thm 3.5}
Let $\mu$ be a positive Radon measure, and let $\chi \in \mathcal{W}^+_M$. The following conditions are equivalent:
\begin{itemize}
    \item[(1)]  there exists a unique function $ \phi \in \mathcal{E}_\chi(\Omega)$ such that $ \mu = (dd^c
\phi)^n$;
\item[(2)]$\chi(\mathcal{E}_\chi(\Omega)) \subset L^1(d\mu)$;
\item[(3)]  there exists a constant $C > 0$ such that
$$ \int_\Omega - \chi \circ \psi d\mu \leq C, $$ for all  $\psi \in \mathcal{E}_0(\Omega)$, 
$  
E_\chi(\psi) \leq 1; $
\item[(4)] there exists a positive constant $A$ such that
$$ \int_\Omega -\chi \circ \psi d\mu \leq A \max (1, E_\chi(\psi)), \; \; \forall \psi \in \mathcal{E}_0(\Omega).
$$
\end{itemize}
\end{theorem}
\begin{proof}
The implication (1) $ \Rightarrow$ (2) is obvious because $\mathcal{E}_\chi(\Omega)$ is a convex cone  and that
 $$ \int_\Omega -\chi \circ u (dd^c v)^n \leq E_\chi(u+v)<+\infty, \; \; \forall u, v \in \mathcal{E}_\chi(\Omega).  $$
 The proof of the implication (2) $\Rightarrow$ (3) resembles to that of (2) $\Rightarrow$ (3) in \cite[Theorem 6]{Ben15}; we repeat it for the reader's convenience. By contradiction, let $(u_j) \in \mathcal{E}_0(\Omega)$ such that $E_\chi(u_j) \leq 1$ and
$$ \int_\Omega -\chi(u_j) d\mu \geq 2^{3Mj}. $$
Consider the function 
$$ u:= \sum_j \frac{1}{2^{2j}} u_j. $$
Since
$ (u<-s) \subset \cup_j (u_j < -2^js), $
we get
$$ \capa_\Omega(u<-s) \leq \sum_j \capa_\Omega(u_j < -2^j s) $$
and therefore
 \begin{align*}
     &\int_0^\infty s^n \chi'(-s) \capa_{\Omega}(u<-s) ds \\
     &\leq \int_0^\infty s^n \chi'(-s) \sum_j \capa_{\Omega}(u_j<-2^js) ds \\
     &=\sum_j 1/2^{nj} \int_0^\infty (2^js)^n \chi'(-s) \capa_{\Omega}(u_j<-2^js) ds. 
 \end{align*}
 By the change of variables $t=2^js$, we obtain 

 \begin{align*}
     &\int_0^\infty s^n \chi'(-s) \capa_{\Omega}(u<-s) ds \\
     &= \sum_j 1/2^{nj} \int_0^\infty t^n \chi'(-2^{-j}t) \capa_{\Omega}(u_j<-t) 2^{-j}dt \\
    (*) &\leq  \sum_j 1/2^{(n+1)j} \int_0^\infty t^n \chi'(-t) \capa_{\Omega}(u_j<-t) dt \\
    &\leq \sum_j 1/2^{(n+1)j}<\infty.
     \end{align*}
$(*)$ is justified by the fact that $\chi$ is concave, so $\chi'$ is non-increasing.
That proves $u \in \tilde{\mathcal{E}}_\chi(\Omega)$ and therefore $u \in \mathcal{E}_\chi(\Omega)$ by Proposition \ref{prop 3.3}.

Since $u \leq 2^{-2j} u_j$, we get $\chi(u) \leq \chi(2^{-2j}u_j)$ because $\chi$ is increasing. It follows from Lemma \ref{lem 3.1} that $\chi(u) \leq 2^{-2Mj} \chi(u_j)$, and therefore 
$$ \int_\Omega -\chi(u) d\mu \geq  2^{-2Mj} \int_\Omega -\chi(u_j) d\mu \geq 2^{Mj}. $$
Which is in contradiction with $(2)$.

We move on to the proof of the implication (3) $\Rightarrow$ (4). 
Consider $\psi \in \mathcal{E}_0(\Omega)$ such that $E_\chi(\psi)\geq 1$. 
If $E_\chi(\psi) \leq 2^{n+1}$, then $E_\chi(1/2 \; \psi) \leq 1$ and therefore
$$ \int_\Omega -\chi(\psi) d\mu \leq 2^M \int_\Omega -\chi(1/2 \; \psi) d\mu \leq 2^M C. $$
Suppose 
$  E_\chi(\psi) \geq 2^{n+1}$, and denote by $\varepsilon = 1/E_\chi(\psi)$. For $f=\chi^{-1}(\varepsilon \chi(\psi))$,
   consider the envelope 
    $$ P(f) = \sup \{ h \in \PSH(\Omega) : h \leq f \}. $$
It is clear that  $P(f) \in \mathcal{E}_0(\Omega)$ since $f$ is upper-semicontinuous and satisfies  $f \geq \psi$. Theorem \ref{thm 2.4} implies
$$ \textit{1}_{\{P(f)<f\}} (dd^c P(f))^n = 0. $$
It thus follows that
\begin{align*}
     E_\chi(P(f)) &= \int_\Omega - \chi(P(f)) (dd^c P(f))^n \\   &= \int_{\{P(f) = f)\}} -\chi(f)(dd^c P(f))^n \\ 
     &= \varepsilon \int_\Omega -\chi(\psi) (dd^c P(f))^n.
\end{align*}
 Applying Theorem \ref{thm 3.4} for $\lambda = 1/2$, we get    
$$
    E_\chi(P(f))) \leq 
    2^n \varepsilon  E_\chi(P(f)) + 2^{M+n}.
$$
That implies
$$ E_\chi(P(f)) \leq \frac{ 2^{M+n}  }{1 - 2^n \varepsilon} \leq 2^{M+n+1},$$
and therefore 
$$ E_\chi(1/2^{M+1} \times P(f)) \leq 1. $$
Furthermore, since $P(f) \leq f = \chi^{-1}(\varepsilon \chi(\psi))$, we obtain 
$$
       \int_\Omega -\chi(\psi) d\mu \leq \varepsilon^{-1}\int_\Omega - \chi(P(f)) d\mu \leq 2^{M(M+1)} C E_\chi(\psi).$$
We conclude that,  for all $\psi \in \mathcal{E}_0(\Omega)$, we have
 $$ 
\int_\Omega -\chi(\psi) d\mu \leq 2^M C + 2^{M(M+1)} C E_\chi(\psi) \leq A\max(1, E_\chi(\psi)),
$$
where we have taken $A=2^{M(M+1)+1} C$.

For the implication (4) $\Rightarrow$ (1), setting $\tilde{\mu} = 1/(2 A)\; \mu$, we have 
$$ \int_\Omega -\chi(\psi) d\tilde{\mu} \leq 1/2 \max(1,E_\chi(\psi)), \; \; \forall \psi \in \mathcal{E}_0(\Omega).  $$
Since 
$$ \limsup_{t \rightarrow +\infty} \frac{\max(1,t)}{2t} = 1/2 <1, $$
we can construct a function $\tilde{u} \in \mathcal{E}_\chi(\Omega)$ such that $\tilde{\mu} = (dd^c \tilde{u})^n$ (the argument is the same as that of the proof of the implication (5) $\Rightarrow$ (1) in \cite[Theorem 6]{Ben15}). The  result follows by taking $u = (2A)^{1/n} \tilde{u}$. 
\end{proof}

\section{Definition of the Monge-Ampère operator $(\omega + dd^c.)^n$}\label{sec 4}
In this section, we study the operator $(\omega + dd^c .)^n$ for a smooth real $(1,1)$-form $\omega$ non necessarily closed. It follows from the work of Bedford and Taylor that this operator is well defined on $\PSH(\Omega,\omega) \cap L^{\infty}(\Omega) $\cite{BT76} (a detailed construction is given in \cite{KN15}). Here we extend this definition to unbounded functions $u \in \mathcal{E}(\Omega,\omega)$.
\subsection{The current $(dd^c .)^k$}
First, we show that the current $(dd^c u)^k$ is well defined for any $u \in \mathcal{E}(\Omega)$ and any $1\leq k\leq n$. The following proposition will be essential for our work.
\begin{prop}\label{prop 4.1}
Fix $p \in \{1,..,n\}$, and let $\alpha$ be a smooth $(p,p)-$form defined in a neighborhood of $\bar\Omega$. One can write
$$ 
\alpha = \sum_{j\in J } f_{j} T_{j},
$$
where
\begin{itemize}
    \item $J$ is a finite set;
    \item $(f_{j})_{j\in J}$ are smooth functions  with complex values;
    \item $ T_{j} = dd^c u_1^j \wedge .. \wedge dd^c u_{p}^j, $  where, for every $ j\in J$ and every $i=1,..,p$, $u_i^j$ is a smooth negative  plurisubharmonic function defined in a neighborhood of $\bar\Omega$. 
\end{itemize}
\end{prop}
\begin{proof}
Write
$$ 
\alpha = i^{p^2} \sum_{|I|=|K|= p} \alpha_{IK} dz_I \wedge d\bar{z}_K. 
$$
It is thus enough to show that, for all $l,k$, the $(1,1)$-form $dz_l\wedge d\bar{z}_k$ can be written as a linear combination of closed positive $(1,1)$-forms with smooth coefficients. It is clear for $l=k$. For $ l \neq k$, it follows from  \cite[Lemma 1.4]{Dem12} that 
 \begin{align*}
 4 dz_l \wedge d\bar{z}_k &= ( dz_l + dz_k ) \wedge \overline{(dz_l + dz_k) } -  ( dz_l - dz_k ) \wedge \overline{(dz_l - dz_k) } \\ &+ i ( dz_l + idz_k ) \wedge \overline{(dz_l + i dz_k) } -i  ( dz_l -i dz_k ) \wedge \overline{(dz_l -i d z_k)}. 
 \end{align*}
Note that 
$$T_j = dd^c u^j_1 \wedge ... \wedge dd^c u_{p}^j,  $$
where the functions $u^j_1,...., u^j_{p}$ are taken as
$$ |z_k|^2 - R, \; |z_l \pm z_k|^2 - R \; \text{and}\; |z_l \pm i z_k|^2 - R, $$
where $ 1 \leq l,k \leq n, $ and  $R >0$ is large enough.
 This proof is thus complete. 
\end{proof}
We now define $(dd^c u)^k$ as a closed positive $(k,k)$-current when $u\in \mathcal{E}(\Omega)$. 
Let $ u \in \mathcal{E}(\Omega)$, and let $\alpha$ be a smooth $(n-k,n-k)$-form defined in a neighborhood of  $\bar\Omega$. We write
$$ \alpha = \sum f_j T_j, $$
where $f_j$ and $T_j$ are as in Proposition \ref{prop 4.1}. 
By \cite[Theorem 4.2]{Ceg04}, $ (dd^c u)^k \wedge T_j $  defines a Radon measure. We define  $(dd^c u)^k \wedge \alpha$ by 
$$ (dd^c u)^k \wedge \alpha = \sum f_j (dd^c u)^k \wedge T_j. $$
By  \cite[Lemma 3.2]{Ceg08}, if $ (v_s)_s \in \mathcal{E}(\Omega)$,  $ v_s \searrow u$, then the sequence $ \left( ( dd^c v_s )^k \wedge T_j\right)_s$ converges weakly to $ ( dd^c u )^k \wedge T_j$. It follows that
$$ 
 ( dd^c v_s )^k \wedge \alpha \longrightarrow  
(dd^c u)^k \wedge \alpha \quad \emph{weakly}.
$$
Suppose now that 
$$ \alpha = \sum f_j T_j = \sum g_l S_l, $$
for $f_j, g_l \in \mathcal{C}^\infty(\bar\Omega)$ and $T_j, S_l$ are as in Proposition \ref{prop 4.1}. We prove that 
$$ \sum f_j (dd^c u)^k \wedge T_j = \sum g_l (dd^c u)^k \wedge S_l. $$
Let $(u_s)_s $ be the standard regularization of $u$.  If $D \subset \subset \Omega$, then $ u_s \in \mathcal{E}(D)$ for $s$ large enough, and  $u_s \searrow u$ on $D$. Since 
$$\sum f_j (dd^c u_s)^k \wedge T_j = \sum g_l (dd^c u_s)^k \wedge S_l, \quad \forall s, $$
 \cite[Lemma 3.2]{Ceg08} gives 
$$ \sum f_j (dd^c u)^k \wedge T_j = \sum g_l (dd^c u)^k \wedge S_l. $$
That proves the following theorem.
\begin{theorem}\label{thm 4.2}
Let $ u \in \mathcal{E}(\Omega)$. For all $k \in \{1,..,n \}$, the current $(dd^c u)^k$ is well defined. Furthermore, if  $ (u_j)_j $  is a decreasing sequence in $ \mathcal{E}(\Omega) $ that converges to $ u$, then the sequence $ \left( ( dd^c u_j)^k \right)_j $ converges weakly  to $ ( dd^c u)^k$.  
\end{theorem}
\subsection{The classes $\mathcal{K}(\Omega,\omega,\phi)$}
Let $\omega$ be a smooth real $(1,1)$-form  defined in a neighborhood of $\bar{\Omega}$. We denote by $\mathcal{P}_\omega(\Omega)$ the set of $\rho \in \mathcal{C}^2(\bar{\Omega}) \cap \PSH(\Omega)$ such that $\rho=0$ on $\partial \Omega$ and $\omega \leq dd^c \rho$.

In the sequel, we denote by $\phi$ a psh maximal function in $\mathcal{E}(\Omega) \cap \mathcal{C}^0(\bar{\Omega})$; the term maximal means that  $(dd^c \phi)^n=0$. Recall that for  $\mathcal{K}(\Omega) \in \{ \mathcal{E}(\Omega)$, $\mathcal{F}(\Omega)$, $\mathcal{E}_p(\Omega)$, $\mathcal{E}_\chi(\Omega),\mathcal{N}(\Omega)\}$, the set $\mathcal{K}(\Omega, \phi)$ is defined by 
$$ u \in \mathcal{K}(\Omega, \phi) \Leftrightarrow u \in \PSH(\Omega) \; \text{and} \; \phi \geq u \geq \phi + \tilde{u}, $$
for $\tilde{u} \in \mathcal{K}(\Omega)$ \cite{Ceg98,Aha07,Ceg08,Ben15}.

Fix $\rho \in \mathcal{P}_\omega(\Omega)$.
We define the set
$\mathcal{K}(\Omega,\omega, \phi)$ by 
$$ u \in  \mathcal{K}(\Omega,\omega, \phi) \Leftrightarrow u \in \PSH(\Omega,\omega) \; \text{and} \;  u + \rho \in \mathcal{K}(\Omega,\phi). $$

\begin{remark}
Note that the last definition does not depend on $\rho$.
Indeed,
let $\rho' \in \mathcal{P}_\omega(\Omega) $. If $u+\rho \in \mathcal{K}(\Omega,\phi)$, then we have 
$$\phi + \tilde{u} \leq u+\rho  \leq \phi$$
 for some $\tilde{u}\in \mathcal{K}(\Omega)$.
 Since $\phi$ is maximal, we get $u+\rho' \leq \phi$. In the other hand 
 $$ u+\rho' \geq u+\rho+\rho' \geq \phi +\tilde{u}+\rho'.$$
That means $u+\rho' \in \mathcal{K}(\Omega,\phi)$.
\end{remark}
\begin{remark}
    Taking $\phi = 0$, the class $\mathcal{K}(\Omega,\omega) := \mathcal{K}(\Omega,\omega,0)$ is a generalization of the usual Cegrell class $\mathcal{K}(\Omega)$. 
\end{remark}
We have the following observation:
\begin{prop}
$$\mathcal{K}(\Omega,\omega,\phi) \subset \mathcal{E}(\Omega,\omega). $$ 
\end{prop}
\begin{proof}
Let $\rho \in \mathcal{P}_\omega(\Omega)$. 
If $u \in \mathcal{K}(\Omega,\omega,\phi)$, then $u+\rho \in \mathcal{K}(\Omega,\phi)$.  It follows that
$$  0 \geq \phi \geq u+\rho \geq \phi + \tilde{u}, $$
for some $\tilde{u} \in \mathcal{K}(\Omega)$. Hence
  $u \in \mathcal{E}(\Omega,\omega)$.
\end{proof}

\subsection{Basic properties} 
In this subsection we study some basic properties of the operator $(\omega + dd^c .)^n$. 
We first show  that the set $\mathcal{E}(\Omega,\omega)$ is always a non-empty subset of $\PSH(\Omega,\omega)$.
\begin{prop}
$$ \PSH^-(\Omega, \omega) \cap L^\infty_{loc}(\Omega) \subset \mathcal{E}(\Omega,\omega). 
$$
\end{prop}
\begin{proof}
Let $ \rho \in \mathcal{P}_\omega(\Omega) $. If $ u \in \PSH^-(\Omega,\omega)\cap L^\infty_{loc}(\Omega), $ then $ u + \rho \in \PSH^-(\Omega) \cap L^\infty_{loc}(\Omega) \subset \mathcal{E}(\Omega) $  and therefore $ u \in \mathcal{E}(\Omega,\omega)$.
\end{proof}
Functions in $\mathcal{E}(\Omega,\omega)$ are not necessarily negative. However, we have the following observation:
\begin{prop}
There is a constant $C$ such that $ u \leq C, $   for every $u \in \mathcal{E}(\Omega,\omega)$.
\end{prop}
\begin{proof}
Let $ \rho \in \mathcal{P}_\omega(\Omega)$.
If $ u \in \mathcal{E}(\Omega,\omega)$, then $ u + \rho \leq 0$ and therefore $ u \leq \sup ( - \rho)$. We take $C = \sup ( - \rho)$.
\end{proof}
Other basic facts are given in the following proposition:
\begin{prop} $\phantom{}$
\begin{itemize}
    \item  $\mathcal{E}(\Omega,0) = \mathcal{E}(\Omega)$.
    \item Let $\omega'$ be another smooth $(1,1)$-form defined in a neighborhood of $\bar{\Omega}$. We have
    $$ \omega\leq \omega' \Rightarrow \mathcal{E}(\Omega,\omega) \subset \mathcal{E}(\Omega,\omega'). $$
    In particular, if $\omega \geq 0$ then $\mathcal{E}(\Omega) \subset \mathcal{E}(\Omega,\omega)$. 
\end{itemize}
\end{prop}
\begin{proof}
For the first statement, we take $\rho = 0$ in the definition of $\mathcal{E}(\Omega,0)$. To prove the second statement, we take $ \rho \in \mathcal{P}_{\omega'}({\Omega})$. 
\end{proof}
\subsection{Definition of the Monge-Ampère operator $( \omega + dd^c.)^n$ on $\mathcal{E}(\Omega,\omega)$} \phantom{} \\
Let $u \in \mathcal{E}(\Omega, \omega)$, and let $ \rho \in \mathcal{P}_\omega(\Omega)$. Write
$$ ( \omega + dd^c u)^n = ( \omega - dd^c \rho + dd^c ( u+ \rho) )^n = \sum_{k=0}^n \binom{n}{k} (-1)^{n-k} ( dd^c (u+\rho))^k \wedge \gamma^{n-k}, $$
where $ \gamma = dd^c \rho - \omega $ is a semi-positive $(1,1)$-form.  From Theorem \ref{thm 4.2},  
$( dd^c (u+\rho) )^k \wedge \gamma^{n-k} $  defines a positive Radon measure. By linearity, we can define 
 the operator $ (\omega + dd^c u)^n$. We have to check that $(\omega+dd^c u)^n$ is independent of $ \rho$: Let $ \rho' \in \mathcal{P}_\omega(\Omega)$. It remains then to prove that 
 \begin{align*}
     & \sum_{k=0}^n \binom{n}{k} (-1)^{n-k} ( dd^c (u+\rho))^k \wedge (dd^c \rho - \omega)^{n-k} \\ =  &\sum_{k=0}^n \binom{n}{k} (-1)^{n-k} ( dd^c (u + \rho'))^k \wedge (dd^c \rho' - \omega)^{n-k}.
 \end{align*}
 Let $D \subset \subset \Omega$, and  let $(u_j)$ denote the standard regularization of $u$. 
 We have \begin{align*}
     & \sum_{k=0}^n \binom{n}{k} (-1)^{n-k} ( dd^c (u_j+\rho))^k \wedge (dd^c \rho - \omega)^{n-k} \\ =  &\sum_{k=0}^n \binom{n}{k} (-1)^{n-k} ( dd^c (u_j + \rho'))^k \wedge (dd^c \rho' - \omega)^{n-k}
 \end{align*}
on $D$ for $j$ large enough. Letting $j \rightarrow +\infty$, the result follows from  Theorem \ref{thm 4.2}. 
\begin{definition}
Let $ u \in \mathcal{E}(\Omega, \omega)$. We define  the operator $ (\omega + dd^c u)^n $ by the  formula
$$ ( \omega + dd^c u)^n  = \sum_{k=0}^n \binom{n}{k} (-1)^{n-k} ( dd^c (u+\rho))^k \wedge (dd^c \rho - \omega)^{n-k}, $$
where $ \rho \in \mathcal{P}_\omega(\Omega)$.
\end{definition}
As a consequence of Theorem \ref{thm 4.2}, we obtain the following result:
\begin{theorem}\label{thm 4.10}
 The Monge-Ampère operator $ (\omega + dd^c .)^n $ is well  defined on the class $\mathcal{E}(\Omega,\omega)$. Furthermore, if $ (u_j)_j \subset {\mathcal{E}}(\Omega,\omega)$ is a decreasing sequence that converges to $u \in \mathcal{E}(\Omega,\omega)$, then the sequence of Radon measures $\left( ( \omega + dd^c u_j )^n \right)_j$ converges weakly to $ ( \omega + dd^c u)^n$.
\end{theorem}
\begin{remark}
    The construction that we gave to the operator $(\omega + dd^c .)^n$ yields more useful information.  It shows, in particular, that the study of the measure $(\omega + dd^c u)^n$, for $u\in \mathcal{E}(\Omega)$,  relies on the study of 
   the mixed Monge-Ampère measures
$$ dd^c u_1 \wedge ... \wedge dd^c u_k, \; \; u_1, ... u_k \in \mathcal{E}(\Omega). $$
This remark will be very helpful in the sequel.
\end{remark}
\section{Degenerate complex Monge-Ampère equations}\label{sec 5} 
In this section, we investigate the existence of solutions to the Dirichlet problem
$$ (\omega + dd^c u)^n = F(u,.) d\mu, $$
for a   positive Radon measure  $\mu$  vanishing on pluripolar sets  and a bounded measurable function  $F: \mathbb{R} \times \Omega \rightarrow [0,+\infty[$   
       which is continuous and non-decreasing  in the first variable.
       
       First, we discuss some properties of the operator $(\omega + dd^c .)^n$.
\subsection{Properties of the operator $(\omega + dd^c .)^n$ on $\mathcal{E}(\Omega,\omega)$}

The following example shows that there is a large difference between the operator $(dd^c.)^n$ and the operator  $(\omega+dd^c .)^n$.
\begin{example}\label{ex 5.1}
Let $\Omega$ denote the unit ball in $\mathbb{C}^2$, and set $\omega =  |z_1|^2 \, dd^c |z_2|^2$. The $(1,1)-$form
$\omega$ is semi-positive and 
$$ dd^c \omega = dd^c |z_1|^2\wedge dd^c |z_2|^2 = 4/\pi \; dV_2, $$
where $dV_2$ is the Lebesgue measure on $\mathbb{C}^2$.
For
\begin{center}
    $u = \max(\log|z|, -1)$ and $v = \max(\log|z|,-1/2)$,
\end{center}
 we have  $ u, v  \in \mathcal{E}_0(\Omega)$, $ u \leq v$ on $\Omega$ and $u = v$ in a neighborhood of $\partial\Omega$. However    
\begin{align*}
    &\int_\Omega (\omega + dd^c u)^2 - (\omega + dd^c v)^2 \\ &= \int_\Omega (dd^c u)^2 - (dd^c v)^2 + 8/\pi \int_\Omega (u - v) d\lambda(z) \\ &= 8/\pi \int_\Omega (u - v) dV_2 < 0.
\end{align*}
This is  contrary to \cite[Corollary 4.3]{BT82} and proves that several inequalities of Cegrell do not extend to the operator $(\omega + dd^c .)^n$.
\end{example}

We denote by $ \mathcal{K}^a(\Omega,\omega)$ the set of $ u \in \mathcal{K}(\Omega,\omega)$ such that the measure $ ( \omega + dd^c u)^n $ vanishes on pluripolar sets. We have the following proposition:
\begin{prop}\label{prop 5.2}
If  $ u \in \mathcal{E}^a(\Omega,\omega)$, then  $ u + \rho \in \mathcal{E}^a(\Omega)$ for all $\rho \in \mathcal{P}_\omega(\Omega)$. 
\end{prop}
\begin{proof}
Let $A$ be a pluripolar subset of $ \Omega$. We have 
\begin{align*}
  0 &=\int_A (\omega + dd^c u)^n \\
    &= \int_A (dd^c (u +\rho))^n + \sum_{k=1}^n (-1)^k C_n^k\int_A (dd^c \rho - \omega)^k \wedge (dd^c(u+\rho))^{n-k}.
\end{align*}
Let $\sigma \in \mathcal{P}_{-\omega}(\Omega)$. \cite[Lemma 4.4]{ACCH08} gives for all $ 1 \leq k \leq n$:
\begin{align*}
     \int_A (dd^c\rho- \omega)^k \wedge (dd^c(u+\rho))^{n-k} &\leq  \int_A (dd^c (\rho+\sigma)) )^k \wedge (dd^c(u+\rho))^{n-k} \\
     &\leq \left( \int_A (dd^c (\rho+\sigma))^n \right)^{k/n} \left( \int_A (dd^c(u+\rho))^n \right)^{(n-k)/n} \\
     &= 0.
\end{align*}
Hence $(dd^c(u+\rho))^n(A)=0$ and $u+\rho \in \mathcal{E}^a(\Omega)$. 
\end{proof}
The operator $(\omega + dd^c .)^n$ verifies a maximum principle that is similar to that of the operator $(dd^c .)^n$ \cite[Theorem 4.1]{KH09} (see also \cite[Theorem 2.2]{BGZ08}).
\begin{theorem}\label{thm 5.3}
If $u, v \in \mathcal{E}(\Omega,\omega)$, then 
$$ \textit{1}_{\{u>v\}} ( \omega + dd^c u)^n = \textit{1}_{\{u>v\}} ( \omega + dd^c \max(u,v))^n. $$
\end{theorem}
\begin{proof}
Fix $\rho \in \mathcal{P}_\omega(\Omega)$.  Write 
$$ ( \omega + dd^c u)^n  = \sum_{k=0}^n \binom{n}{k} (-1)^{n-k} ( dd^c (u+\rho))^k \wedge (dd^c \rho - \omega)^{n-k}, $$
and
$$  (dd^c \rho - \omega)^{n-k} = \sum f_j T_j,  $$
where $f_j$ and $T_j$ are as in Proposition {\ref{prop 4.1}}. By linearity, it suffices to prove that
$$ \textit{1}_{\{ u > v\}} (  dd^c \max(u+\rho,v+\rho))^k \wedge T_j = \textit{1}_{\{ u > v\}} ( dd^c (u+\rho))^k \wedge T_j.  $$
But this is exactly \cite[Theorem 4.1]{KH09}.
\end{proof}
We shall derive  several consequences from the previous theorem.
\begin{corollary}\label{cor 5.4}
Let $u,v \in \mathcal{E}(\Omega,\omega)$, and let $\mu$ be a  positive Radon measure vanishing on pluripolar sets. If 
$$ 
(\omega + dd^c u)^n \geq \mu \; \emph{ and} \;  ( \omega + dd^c v)^n \geq \mu,
$$
then
$$  
(\omega + dd^c \max(u,v))^n \geq \mu. 
$$
\end{corollary}
\begin{proof}
Since the situation is local, there is no loss of generality in assuming $\mu(\Omega)$ finite. By  Theorem  {\ref{thm 5.3}}, 
we have 
\begin{equation*}
     ( \omega + dd^c \max(u,v))^n  \geq  \textit{1}_{\{u>v\}} ( \omega + dd^c u)^n +  \textit{1}_{\{u < v\}} ( \omega + dd^c v)^n \geq \textit{1}_{\{u \neq v\}} \mu. 
\end{equation*}
If $\mu (\{u = v\}) = 0$, then the result follows. 
The proof of \cite[Corollary 1.10]{GZ07} shows that 
$$  \mu(\{ u = v + t \}) = 0, \; \;  \forall t \in \mathbb{R}\setminus I, $$
where $I$ is at most countable. Take $\varepsilon_j \in \mathbb{R}\setminus I$,  $\varepsilon_j \searrow 0$.   We have 
$$ ( \omega  + dd^c \max(u, v + \varepsilon_j))^n \geq \mu. $$ 
It suffices then to let $\varepsilon_j \rightarrow 0$.
\end{proof}

The following corollary  return to J. P. Demailly \cite[Proposition 11.9]{Dem91} in the case $\omega=0$.
\begin{corollary}\label{cor 5.5}
Let $u, v \in \mathcal{E}(\Omega,\omega)$. If  $(\omega + dd^c v)^n$ 
vanishes on pluripolar sets, then 
$$ ( \omega + dd^c \max(u,v))^n \geq \textit{1}_{\{u> v \}} ( \omega + dd^c u)^n + \textit{1}_{\{ u \leq v \}} ( \omega + dd^c v)^n. $$
In particular, if moreover $u \geq v$ then 
$$ \textit{1}_{\{u=v\}} ( \omega + dd^c u)^n \geq  \textit{1}_{\{u=v\}} ( \omega + dd^c v)^n. $$
\end{corollary}
\begin{proof}
The proof of the first affirmation is the same as that of Corollary {\ref{cor 5.4}}. The second one follows easily from the first.
\end{proof}

\subsection{The comparison principle in $\mathcal{N}^a(\Omega,\omega,\phi)$}
 The classical comparison principle (for the operator $(dd^c .)^n$ acting on bounded psh functions) is proven in \cite[Corollary 4.4]{BT82}. This result has been generalized to  the  class $\mathcal{F}^a(\Omega)$ \cite[Theorem 5.15]{Ceg04}, to the large class $\mathcal{N}^a(\Omega,\phi)$ \cite[Theorem 4.4]{Ceg08}
 and to the operator $(\omega + dd^c .)^n$  \cite[Corollary 3.4]{KN15}. We propose the following general version:
\begin{theorem}\label{thm 5.6}
Let $u \in \mathcal{N}^a(\Omega,\omega,\phi)$, and let $v \in \mathcal{E}(\Omega,\omega)$. If
 $$ ( \omega + dd^c u)^n \leq ( \omega + dd^c v)^n \; \;\text{on} \; \{ u<v\}, $$
then $u \geq v$.
\end{theorem}
\begin{proof}
Set $\psi = \max(u,v)$. We have by Corollary {\ref{cor 5.5}} 
\begin{align*}
    (\omega + dd^c \psi)^n &\geq \textit{1}_{\{u \geq v\}} (\omega + dd^c u)^n + \textit{1}_{\{u < v\} }(\omega + dd^c v)^n \\
    &\geq ( \omega + dd^c u)^n.
\end{align*}

Setting $f : = u - \psi$, we  prove that $f = 0$. Define
$$P(f) :=  \sup \{ h \in \PSH(\Omega)\; ; \; h \leq f \}.  $$
It follows from \cite[Proposition 5.1]{BT82} that
 $P(f)^*=P(f)$ almost everywhere, hence $P(f)^* \leq u-v$ a.e. From this we get $P(f)^*+v\leq u$ a.e., hence everywhere because these are psh functions. It thus follows that $P(f)^*=P(f)$. 
Let $\rho \in \mathcal{P}_\omega(\Omega)$.
We have $u+\rho \in \mathcal{N}^a(\Omega,\phi)$ by Proposition \ref{prop 5.2}. Thus we get $\tilde{u} + \phi \leq u+\rho \leq \phi$ for some   $\tilde{u} \in \mathcal{N}^a(\Omega)$. From this we get $\tilde{u} \leq u + \rho - \phi \leq f$ because $\psi + \rho = \max(u+ \rho, v + \rho) \leq \phi$ by maximality of $\phi$. Therefore $\tilde{u} \leq P(f)$ and $P(f) \in \mathcal{N}^a(\Omega)$ according to  \cite[Corollary 3.14]{Ceg08}. 

Setting $D = \{ P(f) = f \}$, since $P(f) \leq f$, we obtain by Corollary {\ref{cor 5.5}}
$$ \textit{1}_D ( \omega + dd^c ( \psi + P(f)))^n \leq \textit{1}_D ( \omega + dd^c u)^n \leq \textit{1}_D ( \omega + dd^c \psi)^n. $$
It  follows that $\textit{1}_D (dd^c P(f))^n=0$, and hence $(dd^c P(f))^n=0$ according to  Theorem \ref{thm 2.4}. It follows from \cite[Lemma 3.12]{Ceg08}
that $P(f) = 0$. Thus, we conclude that $f = 0$ and therefore $u \geq v$. 
\end{proof}
Theorem D in the introduction follows immediately from the last theorem. 
\begin{corollary}[Theorem D in the introduction]\label{cor 5.7}
  Let $\mu \leq \nu$ be positive Radon measures vanishing on pluripolar  sets. 
Assume $u \in \mathcal{N}(\Omega,\omega,\phi)$ and  $v \in \mathcal{E}(\Omega,\omega)$ are such that  $v \leq \phi$ in $\partial \Omega$,
$$ (\omega + dd^c u)^n = F(u,.) d\mu \; \text{and} \;  (\omega + dd^c v)^n = F(v,.) d\nu. $$
 Then $u \geq v$. 
\end{corollary}
\begin{proof}
     It follows from Corollary \ref{cor 5.5} that 
\begin{align*}
    (\omega + dd^c \max(u,v))^n &\geq \textit{1}_{\{u> v \}} ( \omega + dd^c u)^n + \textit{1}_{\{ u \leq v \}} ( \omega + dd^c v)^n \\
    &\geq \textit{1}_{\{u> v \}} F(u,.) d\mu  + \textit{1}_{\{ u \leq v \}} F(v,.) d\mu \\
    &= F(\max(u,v),.) d\mu \\
    &\geq F(u,.) d\mu = (\omega + dd^c u)^n,
\end{align*}
where the last inequality follows from the fact that the function $F$ is non-deceasing in the first variable. 
 Theorem \ref{thm 5.6} implies  $u \geq v$.
\end{proof}
\subsection{Solution to the Dirichlet problem}
In this subsection, we study the main question of solving the complex Monge-Ampère type equations in $\mathcal{K} (\Omega,\omega,\phi)$, where $\phi \in \mathcal{E}(\Omega) \cap \mathcal{C}^0(\Omega)$ is maximal and  $\mathcal{K}  \in \{  \mathcal{F}^a$, $\mathcal{E}_p$, $\mathcal{E}_\chi$, $\mathcal{N}^a\}$. The following theorem extends the result of Czy{\.z} \cite{Cz09} to the operator $(\omega + dd^c.)^n$ for any smooth real $(1,1)$-form $\omega$.
\begin{theorem}[Theorem C in the introduction]\label{thm 5.8}
 Let $\mu$ be a positive Radon measure vanishing on pluripolar sets, and consider a bounded measurable function  $F: \mathbb{R} \times \Omega \rightarrow [0,+\infty[$   
       which is continuous and non-decreasing  in the first variable. 
     If $\mu \leq  (dd^c w)^n$ for some $w \in \mathcal{K}(\Omega)$,
      then there is a uniquely determined $u \in \mathcal{K}(\Omega,\omega,\phi)$ such that   
      $$(\omega + dd^c u)^n = F(u,.) d\mu.$$
\end{theorem}

Before giving the proof of the last theorem,  we  need to generalize
the subsolution theorem of Ko{\l}odziej and Nguyen \cite[Theorem 4.1]{KN23a}  to the case when $\omega$ is merely real.
\begin{lemma}[The subsolution theorem of Ko{\l}odziej and Nguyen]
Let $\varphi \in \mathcal{C}^0(\partial\Omega)$ and let $\mu$ be a positive Radon measure. Suppose  $\mu \leq (dd^c v)^n$ for a bounded psh function $v$ such that $v = 0$ on $\partial \Omega$. Then, there exists $u \in \PSH(\Omega,\omega) \cap L^\infty(\Omega)$, $u = \varphi$ on $\partial \Omega$ and such that 
$$ 
(\omega + dd^c u)^n = \mu. $$ 
\end{lemma}
\begin{proof}
Let $\sigma \in \mathcal{C}^\infty(\bar{\Omega})$ be such that $dd^c \sigma > -\omega$. Applying \cite[Theorem 2.3]{KN23b} to $\omega + dd^c\sigma$ and $F(.+ \sigma,.)$ yields $w \in \PSH(\Omega, \omega + dd^c \sigma)$, $w = \varphi - \sigma$ on $\partial \Omega$ and such that 
$$ ( \omega + dd^c \sigma + dd^c w)^n = F(w+\sigma, .) d\mu. $$
Thus, it  suffices to take $u = w +\sigma$.
\end{proof}
\begin{remark}
The same proof shows that
\cite[Corollary 3.4]{KN15} and \cite[Proposition 2.2]{KN23b} hold  when  $\omega$ is merely real. 
\end{remark}
We move on to  the proof of Theorem \ref{thm 5.8}.
\begin{proof}[Proof of Theorem {\ref{thm 5.8}}]
By \cite[Theorem 5.11]{Ceg04}, there is $\psi \in \mathcal{E}_0(\Omega)$ and $f \in L^1_{loc}((dd^c \psi)^n)$ such that
$\mu = f(dd^c \psi)^n$.
It follows from \cite[Theorem 2.3]{KN23b} that,
for every $j$, there exists $u_j \in \PSH(\Omega,\omega) \cap L^\infty(\Omega)$  such that $u_j =\phi$ on $\partial \Omega$ and  
$$ (\omega + dd^c u_j)^n = F(u_j,.) \min(f,j) (dd^c \psi)^n. $$
The sequence $(u_j)_j$ is decreasing by \cite[Proposition 2.2]{KN23b}.

Fix $\rho \in \mathcal{P}_\omega(\Omega)\cap \mathcal{P}_{-\omega}(\Omega)$.
 It follows from \cite[ Theorem 4.14]{ACCH08} that there is $h \in \mathcal{K}(\Omega,\phi)$ such that 
$$ F(\phi - \rho,.) d\mu =  (dd^c h)^n, $$
because $F$ is bounded and $\mu \leq (dd^c w)^n$ for a $w \in \mathcal{K}(\Omega)$.
Using the fact that the function $F$ is non-decreasing  in the first variable, we have
$$ (\omega + dd^c u_j)^n \leq F(u_j,.) d\mu \leq F(\phi - \rho,.) d\mu \leq (\omega + dd^c (h+\rho))^n. $$
Corollary \ref{cor 5.4} gives
$$ (\omega + dd^c u_j)^n \leq (\omega + dd^c 
\max(u_j,h+\rho))^n. $$
Since $u_j = \phi \geq h+\rho$  on $\partial \Omega$, 
it follows from \cite[Corollary 3.4]{KN15} that $u_j \geq h+\rho$ everywhere in $\Omega$, and hence the function $u := \lim_j u_j$ belongs to $\mathcal{K}(\Omega,\omega,\phi)$. Theorem \ref{thm 4.10} gives 
  $$ (\omega + dd^c u)^n = \lim (\omega + dd^c u_j)^n \; \; \text{weakly}. $$
  Since $F$ is bounded and continuous in the first variable, we have by Lebesgue’s dominated convergence theorem
$$ (\omega + dd^c u_j)^n = F(u_j,.) \min (f,j) (dd^c \psi)^n \rightarrow F(u,.) d\mu, $$
in the weak sense of measure. 
 That implies
$$ (\omega + dd^c u)^n = F(u,.) d\mu.   $$
The solution $u$ is uniquely determined by Corollary \ref{cor 5.7}.
\end{proof}
As a consequence, taking $F=1$ in the last theorem yields the following description of the range of the operator $(\omega + dd^c .)^n$. 
\begin{corollary}
 Let $\mu$ be a positive Radon measure. The following statements are equivalent: 
\begin{itemize}
    \item[(i)] the equation $\mu = (dd^c u)^n$ has a solution $u \in \mathcal{K}(\Omega)$;
    \item[(ii)] the equation  $ (\omega + dd^c v)^n = \mu$ has  a unique solution $v \in \mathcal{K}(\Omega,\omega)$;
    \item[(iii)] the equation  $ (\omega + dd^c \varphi)^n = \mu$ has  a unique solution $\varphi \in \mathcal{K}(\Omega,\omega, \phi)$.
\end{itemize}
\end{corollary}
\begin{proof}
    The implications (i) $\Rightarrow$ (ii) and (i) $\Rightarrow$ (iii) follows from Theorem \ref{thm 5.8}. The equivalence is achieved by \cite[Theorem 8.2]{Ceg98}, \cite[Theorem 3.9]{Ceg08}, \cite[Theorem 4.14]{ACCH08} and \cite[Theorem 11]{Ben15}.
\end{proof}

\end{document}